\documentclass[12pt, a4paper]{birkjour}
\usepackage[utf8]{inputenc}

\usepackage{amsmath,amssymb,amsthm,caption,interval,mathrsfs,mathtools,tikz-cd,xcolor}
\usepackage{blkarray}
\usepackage{hyperref}

\DeclareMathOperator{\minn}{min}

\newtheorem{counter}{}

\theoremstyle{definition}

\newtheorem{example}[counter]{Example}
\newtheorem{remark}[counter]{Remark}

\theoremstyle{plain}
\newtheorem{corollary}[counter]{Corollary}
\newtheorem{lemma}[counter]{Lemma}
\newtheorem{proposition}[counter]{Proposition}
\newtheorem{theorem}[counter]{Theorem}

\begin{document}

  \title{On Abstract Spectral Constants}
		\author{Felix L.~Schwenninger}
  \address{Department of Applied Mathematics\\
        University of Twente, P.O.~Box 217\\7500 AE Enschede, The Netherlands}
       \email{f.l.schwenninger@utwente.nl}
        \author{Jens de Vries}
        \address{Department of Applied Mathematics\\
        University of Twente, P.O.~Box 217\\ 
        7500 AE Enschede, The Netherlands}
         \email{j.devries-4@utwente.nl}

\thanks{The second named author has been supported by the Dutch Research Council (NWO) grant OCENW.M20.292.}

	\begin{abstract}
 We prove bounds for a class of unital homomorphisms arising in the study of spectral sets, by involving extremal functions and vectors. These are used to recover three celebrated results on spectral constants by Crouzeix--Palencia, Okubo--Ando and von Neumann in a unified way and to refine a recent result by Crouzeix--Greenbaum. 
	\end{abstract}
\maketitle

\section{Introduction}
Let $M$ be a bounded linear operator on a complex Hilbert space $H$ and $W$ a bounded subset of $\mathbb{C}$. Estimates of the form
\begin{align}\label{spectral estimate}
    \|p(M)\| \leq \kappa\sup_{z\in W} |p(z)|, 
\end{align}
for some positive constant $\kappa$ and all complex polynomials $p$, 
are classical, and appear for instance in von Neumann's inequality or Crouzeix's conjecture. Such $W$ is called a \textit{$\kappa$-spectral set}\footnote{In the literature a $\kappa$-spectral set is typically instead defined by requiring that (\ref{spectral estimate}) holds for all rational functions $p$ with poles off $W^{-}$ and such that the spectrum of $M$ is contained in $W$. For the applications we have in mind, considering polynomials is however sufficient.} for $M$ and $\kappa$ is referred to as a \textit{spectral constant} for $M$. In this article we present an abstract framework to derive spectral constants in a rather direct way. In particular, this approach allows for unified proofs of several well-known results in this context.

    More precisely, we study norms of unital bounded homomorphisms 
    \begin{align*}
        \gamma\colon A\to \mathscr{L}(H),
    \end{align*}
    where $A$ is unital uniform algebra on some compact Hausdorff space $K$---that is, a closed unital subalgebra of the continuous complex-valued functions $C(K)$---and $\mathscr{L}(H)$ denotes the space of bounded linear operators on $H$.
	   Our interest in norms of such mappings stems from  spectral constants appearing in functional calculi for holomorphic functions. Let $\Omega\subset\mathbb{C}$ be a bounded open subset with smooth boundary. If we write $A(\Omega)$ for the space of all continuous functions from the closure $\Omega^{-}$ to $\mathbb{C}$ that are holomorphic on $\Omega$, then $A(\Omega)$ can be viewed as a unital uniform algebra acting on the boundary $\Omega^{\partial}$ by the maximum modulus principle. Suppose that $M$ is a bounded linear operator on $H$ whose spectrum is contained in $\Omega$. The latter guarantees that for each $f\in A(\Omega)$ the integral
	\begin{align*}
		\gamma(f):=\frac{1}{2\pi i}\int_{\Omega^{\partial}}f(\sigma)(\sigma-M)^{-1} \ \mathrm{d}\sigma
	\end{align*} 
	is well-defined. For the rest of the paper we will refer to the linear map $\gamma\colon A(\Omega)\to\mathscr{L}(H)$ as the (Dunford--Riesz) \textit{functional calculus} of $M$ on $\Omega$. It is well-known that $\gamma$ is a unital bounded homomorphism such that $\gamma(p)=p(M)$ for polynomials $p$, see e.g.\ \cite{dunfordschwartzVol1}, where $p$ is interpreted as a function on $\Omega$ in the obvious way. Clearly, bounds on the norm of $\gamma$ provide spectral constants for $M$. We note in passing that the above homomorphism actually extends to the algebra of holomorphic functions on $\Omega$, equipped with the topology of uniform convergence on compact sets, see e.g.\ \cite{haase2018lectures}. 
	
	Let us now return to the general setting. Note that we always have the lower bound $\|\gamma\|\geq1$ as $\gamma(1)=1$. Our starting point for finding an upper bound for the norm of $\gamma$ is the assumption that the norm of $\gamma$ is attained. More precisely, a couple $(f_{0},x_{0})\in A\times H$ with $\|f_{0}\|_{K}=1$ and $\|x_{0}\|=1$ is called an \textit{extremal pair} for $\gamma$ if $\|\gamma\|=\|\gamma(f_{0})x_{0}\|$. While the existence of extremal pairs may seem restrictive, the fact that their existence is e.g.\ guaranteed for functional calculi $\gamma$ on finite-dimensional spaces $H$ if $\Omega$ is convex, see \cite[Theorem 2.1]{crouzeix2004bounds} is typically sufficient to find spectral constants for general $H$. The main advantage of extremal pairs $(f_{0},x_{0})$ is that they translate the operator norm $\|\gamma\|$ into a Hilbert space norm $\|\gamma(f_{0})x_{0}\|$, which is a fruitful tool that we shall heavily exploit. Moreover, we use recently introduced \textit{extremal measures} associated with an extremal pair, \cite{bickel2020crouzeix}. We emphasize that we are merely working with the existence of an extremal pair rather than assuming explicit knowledge of $(f_{0},x_{0})$ which would trivially yield the norm of $\gamma$.

 Our research builds on techniques for deriving spectral constants by employing contour integral representations of the functional calculus. These were developed in the last decades, including the influential paper by Delyon--Delyon \cite{delyon1999generalization}, and Crouzeix's celebrated result  showing that the numerical range of any bounded linear operator is an $11.08$-spectral set, see \cite{crouzeix2004bounds} as well as \cite{crouzeix2007numerical}. At the same time Putinar--Sandberg \cite{putinar2005skew} also studied properties of the operator-valued \textit{double-layer potential}, which is crucial in all these references to derive an norm bound on a related ``symmetrized functional calculus''\footnote{This mapping, however, fails to be multiplicative in general.}
 \begin{align}\label{symmmetrizedfunctionalcalculus}
    f\mapsto\gamma(f)+\gamma(\Phi(f))^{*},
  \end{align}
  where $\Phi\colon A(\Omega)\to A(\Omega)$ refers to a suitable antilinear transformation. See also the recent work on finite-dimensional dilations in \cite[Theorem 1.6]{hartz2021dilation}.
The more recent refinement of Crouzeix's theorem due to Crouzeix--Palencia \cite{crouzeix2017numerical}, see also \cite{ransford2018remarks}, exploits more properties of the double-layer potential, and of the specific $\Phi$, yielding that the numerical range is even a $(1+\sqrt{2})$-spectral set. As shown in \cite{ransford2018remarks}, their proof technique cannot yield any better constant, which implies that yet more structure of the problem has to be used to possibly show Crouzeix's conjecture stating that the numerical range is a $2$-spectral set. 

Under additional assumptions, the bound $1+\sqrt{2}$ can indeed be lowered. Indeed, if $\Omega$ is a disk containing the numerical range, then the optimal spectral constant $2$ can be proved by slightly diverging from Crouzeix--Palencia's proof and instead exploiting properties of extremal pairs, see \cite{caldwell2018some}. The result, however, is much older as it traces back to Okubo--Ando and dilation theory, see \cite{okubo1975constants}. Yet, the fact that such contour integral methods can be used to prove sharp spectral constants seems remarkable. Another instance of this is Delyon--Delyon's reasoning for von Neumann's inequality \cite{delyon1999generalization}, see also \cite[Section 4]{crouzeix2019spectral}.

One finding of our current work is that all of the mentioned results indeed follow directly from one abstract theorem. Thereby we aim to provide a very clean mutual relation of existing results around the conjecture. Surely, this is not the first attempt to study the approach to Crouzeix's conjecture in an abstract fashion.  In \cite{ransford2018remarks} and \cite{ostermann2020abstract} an abstract framework was used to discuss limitations of the proof ingredients used by Crouzeix--Palencia and how to possibly overcome those. On the other hand, the links between von Neumann's inequality and statements about the numerical range have also been studied in the past, see \cite{okubo1975constants} and \cite{Pagacz2020}.

Finally, we remark that spectral constants naturally appear in numerical linear algebra, as they can be used to bound residuals appearing in Krylov subspace approximation methods such as the GMRES algorithm, see e.g.\ \cite[Section 7.1]{crouzeix2019spectral}. 
 
\section{Main Result}
	First we discuss some properties of extremal pairs.
	In \cite[Theorem 5.1]{caldwell2018some} the following proposition was proved for $\gamma$ associated to functional calculi. We slightly adapt their argument, and thereby showing that the result also holds for general $\gamma$. 
	\begin{proposition}\label{orthogonal}
		Let $A$ be a unital uniform algebra acting on a compact Hausdorff space $K$, $H$ a Hilbert space and $\gamma\colon A\to \mathscr{L}(H)$ a unital bounded homomorphism. If $(f_{0},x_{0})$ is extremal for $\gamma$, then $(\|\gamma\|^{2}-1)\langle\gamma(f_{0})x_{0},x_{0}\rangle=0$. In particular, if $\|\gamma\|>1$, then $\langle\gamma(f_{0})x_{0},x_{0}\rangle=0$.
	\end{proposition}
	\begin{proof}
		Fix $\lambda\in\mathbb{C}$ with $|\lambda|=1$ such that $\lambda^{*}\langle\gamma(f_{0})x_{0},x_{0}\rangle=|\langle\gamma(f_{0})x_{0},x_{0}\rangle|$. For each $n\in\mathbb{N}$ we define
		\begin{align*}
			a_{n}:=\frac{\lambda}{2n},\qquad g_{n}:=\frac{f_{0}-a_{n}}{1-a_{n}^{*}f_{0}}.
		\end{align*}
		Since $\|g_{n}\|_{K}\leq1$, it follows that
		\begin{align*}
			\|(\gamma(f_{0})-a_{n})x_{0}\|&=\|(\gamma(f_{0})-a_{n})(1-a_{n}^{*}\gamma(f_{0}))^{-1}(1-a_{n}^{*}\gamma(f_{0}))x_{0}\|\\
			&=\|\gamma(g_{n})(1-a_{n}^{*}\gamma(f_{0}))x_{0}\|\\
			&\leq\|\gamma\|\|(1-a_{n}^{*}\gamma(f_{0}))x_{0}\|.
		\end{align*}
		Squaring both sides and using that $(f_{0},x_{0})$ is extremal for $\gamma$ yields
		\begin{align*}
			\|\gamma\|^{2}-\frac{|\langle\gamma(f_{0})x_{0},x_{0}\rangle|}{n}+\frac{1}{4n^{2}}\leq\|\gamma\|^{2}\bigg(1-\frac{|\langle\gamma(f_{0})x_{0},x_{0}\rangle|}{n}+\frac{\|\gamma\|^{2}}{4n^{2}}\bigg)
		\end{align*}
		and therefore, after rearranging,
		\begin{align*}
			(\|\gamma\|^{2}-1)|\langle\gamma(f_{0})x_{0},x_{0}\rangle|\leq\frac{\|\gamma\|^{4}-1}{4n}.
		\end{align*}
		The desired result follows as the last inequality holds for all $n\in\mathbb{N}$.
	\end{proof}
	We will use the following elementary fact in the proof of the subsequent proposition.
	\begin{lemma}\label{C*-like identity}
		Let $M$ be a bounded linear operator on a Hilbert space $H$. If $x_{0}$ is a unit vector in $H$ on which $M$ attains its norm, that is $\|M\|=\|Mx_{0}\|$, then $M^{*}Mx_{0}=\|M\|^{2}x_{0}$.
	\end{lemma}
	\begin{proof}
		One readily verifies that
		\begin{align*}
			\|M^{*}Mx_{0}-\|M\|^{2}x_{0}\|^{2}=\|M^{*}Mx_{0}\|^{2}-\|M\|^{4}\leq0
		\end{align*}
		and this completes the proof.
	\end{proof}
	The next proposition can be found in \cite[Theorem 4.5]{bickel2020crouzeix} for functional calculi $\gamma$. Using elementary operator theory from the book \cite{paulsen2002completely}, we present a new proof that also works for general $\gamma$.
		\begin{proposition}\label{existence extremal measure}
		Let $A$ be a unital uniform algebra acting on a compact Hausdorff space $K$, $H$ a Hilbert space and $\gamma\colon A\to \mathscr{L}(H)$ a unital bounded homomorphism. If $(f_{0},x_{0})$ is extremal for $\gamma$, then there is a Radon measure $\mu$ on $K$ such that
		\begin{align}\label{extremal measure}
			\langle\gamma(f)x_{0},x_{0}\rangle=\int_{K}f(z) \ \mathrm{d}_{\mu}z
		\end{align}
		for all $f\in A$. Moreover, $\mu$ is a probability measure.
	\end{proposition}
	\begin{proof}
		Consider the functional $\omega\colon A\to\mathbb{C}$ given by $\omega(f):=\langle\gamma(f)x_{0},x_{0}\rangle$ for $f\in A$. Note that $\omega(1)=1$. We claim that $\omega$ is contractive. Take $f\in A$ with $\|f\|_{K}\leq1$. Since $\|f_{0}f\|_{K}\leq1$ and $(f_{0},x_{0})$ is extremal for $\gamma$, it follows from Lemma \ref{C*-like identity} that
		\begin{align*}
			\|\gamma(f_{0})\|^{2}|\omega(f)|&=|\langle \gamma(f)x_{0},\|\gamma(f_{0})\|^{2}x_{0}\rangle|\\
			&=|\langle\gamma(f)x_{0},\gamma(f_{0})^{*}\gamma(f_{0})x_{0}\rangle|\\
			&=|\langle\gamma(f_{0}f)x_{0},\gamma(f_{0})x_{0}\rangle|\\
			&\leq\|\gamma(f_{0})\|^{2}
		\end{align*}
		and therefore $|\omega(f)|\leq1$ as desired. By \cite[Proposition 2.12]{paulsen2002completely} the functional $\hat{\omega}\colon A+A^{*}\to\mathbb{C}$ given by $\hat{\omega}(f+g^{*}):=\omega(f)+\omega(g)^{*}$ for all $f,g\in A$ is well-defined and is the unique positive extension of $\omega$ to the operator subsystem $A+A^{*}$ of the C*-algebra $C(K)$. By a Hahn--Banach argument, see also \cite[Exercise 2.10]{paulsen2002completely}, we can extend 
		$\hat{\omega}$ to a positive functional on $C(K)$, which we also denote by $\hat{\omega}$. So there exists a unique Radon measure $\mu$ on $K$ such that
		\begin{align*}
			\hat{\omega}(h)=\int_{K}h(z) \ \mathrm{d}_{\mu}z
		\end{align*}
		for all $h\in C(K)$. In particular, 
		\begin{align*}
			\mu(K)=\int_{K}1 \ \mathrm{d}_{\mu}z=\hat{\omega}(1)=\omega(1)=1,
		\end{align*}
		which shows that $\mu$ is a probability measure.
	\end{proof}
	A Radon measure $\mu$ on $K$ satisfying (\ref{extremal measure}) is called an \textit{extremal measure} for the extremal pair $(f_{0},x_{0})$. Any such $\mu$ is linked to $f_{0}$ by the relation
	\begin{align*}
		(\|\gamma\|^{2}-1)\int_{K}f_{0}(z) \ \mathrm{d}_{\mu}z=0,
	\end{align*}
	which follows directly from Proposition \ref{orthogonal}. In particular, if $\|\gamma\|>1$, then $f_{0}$ and $1$ are mutually orthogonal in the Hilbert space $L^{2}(K,\mu)$.
	\begin{remark}
		The proof of Proposition \ref{existence extremal measure} tells us that an extremal measure $\mu$ for $(f_{0},x_{0})$ is unique if $A$ is \textit{Dirichlet}, that is if $A+A^{*}$ is dense in $C(K)$. The authors of \cite{gamelin1971pointwise} determined exactly the topological conditions on $\Omega$ that ensure that $A(\Omega)$ is Dirichlet.
	\end{remark}
	For any bounded antilinear map $\Phi\colon A\to A$ we define a bounded linear map $\gamma_{\Phi}\colon A\to \mathscr{L}(H)$ by the formula
	\begin{align*}
		\gamma_{\Phi}(f):=\gamma(f)+\gamma(\Phi(f))^{*}
	\end{align*}
	for all $f\in A$. Moreover, we write $A'$ for the dual space of $A$. We canonically embed $A'$ in the Banach space of bounded linear maps from $A$ to $\mathscr{L}(H)$.
	
	We are now ready to discuss the main theorem.

	\begin{theorem}\label{main}
		Let $A$ be a unital uniform algebra acting on a compact Hausdorff space $K$, $H$ a Hilbert space, $\gamma\colon A\to \mathscr{L}(H)$ a unital bounded homomorphism and $\Phi\colon A\to A$ a bounded antilinear map. 
        If $\|\gamma\|>1$ and $(f_{0},x_{0})$ is extremal for $\gamma$, then
		\begin{align*}
			\|\gamma\|\leq\bigg(\frac{1}{2}\inf_{\omega\in A'}\|\gamma_{\Phi}-\omega\|\bigg)+\sqrt{\bigg(\frac{1}{2}\inf_{\omega\in A'}\|\gamma_{\Phi}-\omega\|\bigg){}^{2}+|\langle\gamma(\Phi(f_{0})f_{0})x_{0},x_{0}\rangle|}.
		\end{align*}
	\end{theorem}
	\begin{proof} It follows from Proposition \ref{orthogonal} that $\langle\gamma(f_{0})x_{0},x_{0}\rangle=0$ and therefore
		\begin{align*}
			\|\gamma\|^{2}&=\|\gamma(f_{0})x_{0}\|^{2}\\
			&=|\langle\gamma(f_{0})x_{0},\gamma(f_{0})x_{0}\rangle|\\
			&\leq|\langle\gamma(f_{0})x_{0},(\gamma_{\Phi}(f_{0})-\omega(f_{0}))x_{0}\rangle|+|\langle\gamma(\Phi(f_{0})f_{0})x_{0},x_{0}\rangle|\\
   &\qquad+|\omega(f_{0})||\langle\gamma(f_{0})x_{0},x_{0}\rangle|\\
			&\leq\|\gamma\|\|\gamma_{\Phi}-\omega\|+|\langle\gamma(\Phi(f_{0})f_{0})x_{0},x_{0}\rangle|
		\end{align*}
		for all $\omega\in A'$. This proves the desired inequality.
	\end{proof}
	\begin{remark}
		The infimum appearing in Theorem \ref{main} is by definition the shortest distance from $\gamma_{\Phi}$ to the subspace $A'$ in the Banach space of bounded linear maps from $A$ to $\mathscr{L}(H)$. Moreover, for each $\omega\in A'$ the map $\gamma_{\Phi}-\omega$ is a rank-one perturbation of $\gamma_{\Phi}$.
	\end{remark}
 The following example is taken from \cite{ransford2018remarks}, where it was used to prove that $\kappa:=1+\sqrt{2}$ is a sharp absolute upper bound for the operator norms of functional calculi $\gamma\colon A(\Omega)\to\mathscr{L}(H)$ for which there exists a contractive antilinear map $\Phi\colon A(\Omega)\to A(\Omega)$ with $\|\gamma_{\Phi}\|\leq2$. 
 
 In addition to the extremal function $f_0$ given in \cite{ransford2018remarks}, we provide an extremal vector $x_0$ and an expression for the associated extremal measure $\mu$. The example also shows that the estimate in Theorem \ref{main} cannot be improved.
\begin{example}\label{RS example}
	For $j=0,1$ let $\Omega_{j}$ be the open disk in $\mathbb{C}$ with center $j$ and radius $1/4$. Consider the set $\Omega:=\Omega_{0}\cup\Omega_{1}$ and the matrix
	\begin{align*}
		M=\begin{pmatrix}
			1&1\\
			0&0
		\end{pmatrix}.
	\end{align*}
	If $\gamma\colon A(\Omega)\to\mathscr{L}(\mathbb{C}^{2})$ is the functional calculus of $M$ on $\Omega$, then one readily verifies that
	\begin{align*}
		\gamma(f)=\begin{pmatrix}
			f(1)&f(1)-f(0)\\
			0&f(0)
		\end{pmatrix}
	\end{align*}
for all $f\in A(\Omega)$. It is not difficult to see that the pair $(f_{0},x_{0})$ defined by
\begin{align*}
f_{0}(z):=\begin{cases}
	-1&z\in\Omega_{0}^{-}\\
	+1&z\in\Omega_{1}^{-}
\end{cases},\qquad x_{0}:=\begin{pmatrix}
	(2-\sqrt{2})^{1/2}/2\\
	(2+\sqrt{2})^{1/2}/2
\end{pmatrix}
\end{align*}
is extremal for $\gamma$ and that 
\begin{align*}
\|\gamma\|=\|\gamma(f_{0})x_{0}\|=1+\sqrt{2}.
\end{align*}
If $\mu$ is an extremal measure for $(f_{0},x_{0})$, then a direct computation yields
\begin{align}\label{extremal measure example}
	\int_{\Omega^{\partial}}f(z) \ \mathrm{d}_{\mu}z=\langle\gamma(f)x_{0},x_{0}\rangle=\frac{f(0)+f(1)}{2}
\end{align}
for all $f\in A(\Omega)$. The antilinear map $\Phi\colon A(\Omega)\to A(\Omega)$ given by
\begin{align*}
	\Phi(f)(z):=\begin{cases}
		-f(0)^{*}&z\in\Omega_{0}^{-}\\
		-f(1)^{*}&z\in\Omega_{1}^{-}
	\end{cases}
\end{align*}
for $f\in A(\Omega)$ is contractive. Moreover, it satisfies $\|\gamma_{\Phi}(f)\|=|f(1)-f(0)|$ for all $f\in A(\Omega)$ and therefore $\|\gamma_{\Phi}\|=2$. Because $\Phi(f_{0})f_{0}=-1$, we furthermore have $\langle\gamma(\Phi(f_{0})f_{0})x_{0},x_{0}\rangle=-1$ by equation (\ref{extremal measure example}). So Theorem \ref{main} gives
\begin{align*}
	\|\gamma\|\leq\bigg(\frac{1}{2}\inf_{\omega\in A(\Omega)'}\|\gamma_{\Phi}-\omega\|\bigg)+\sqrt{\bigg(\frac{1}{2}\inf_{\omega\in A(\Omega)'}\|\gamma_{\Phi}-\omega\|\bigg){}^{2}+1}\leq1+\sqrt{2},
\end{align*}
which implies that the infimum is attained at $\omega=0$ and is equal to $\|\gamma_{\Phi}\|=2$. 
\end{example}
	\section{Applications to Functional Calculi}
	
	Suppose that $\Omega$ is a bounded open subset of $\mathbb{C}$ with smooth boundary. The boundary $\Omega^{\partial}$ is diffeomorphic to a finite disjoint union of circles, see e.g.\ \cite[Appendix]{milnor1997topology}. Fix a Hilbert space $H$ and a bounded linear operator $M$ thereon with spectrum inside of $\Omega$. Recall that the \textit{functional calculus} $\gamma\colon A(\Omega)\to\mathscr{L}(H)$ of $M$ on $\Omega$ is defined by
	\begin{align*}
		\gamma(f):=\frac{1}{2\pi i}\int_{\Omega^{\partial}}f(\sigma)(\sigma-M)^{-1} \ \mathrm{d}\sigma
	\end{align*} 
	for $f\in A(\Omega)$. A positive number $\kappa$ is called a \textit{spectral constant} for $M$ on $\Omega$ if $\|\gamma\|\leq\kappa$.
	
	Let $v\colon\Omega^{\partial}\to\mathbb{C}$ be the map that sends any boundary point in $\Omega^{\partial}$ to the outward unit normal vector of $\Omega$ at that point. We define the \textit{double-layer potential} $P\colon\Omega^{\partial}\to\mathscr{L}(H)$ of $M$ on $\Omega$ by
	\begin{align*}
		P(\sigma):=\bigg(\frac{v(\sigma)}{2\pi}(\sigma-M)^{-1}\bigg)+\bigg(\frac{v(\sigma)}{2\pi}(\sigma-M)^{-1}\bigg){}^{*}
	\end{align*}
	for $\sigma\in\Omega^{\partial}$. Observe that the values of the double-layer potential are Hermitian.
	
	Let $c$ be a connected component of $\Omega^{\partial}$. Write $\ell_{c}$ for the total arc length of $c$ and suppose that $\sigma_{c}\colon\interval{0}{\ell_{c}}\to\mathbb{C}$ is an arc length parametrization of $c$ with the orientation induced by the positive orientation on $\Omega$. So by choice of orientation we have 
	\begin{align*}
		v(\sigma_{c}(s))=\frac{\sigma_{c}'(s)}{i}.
	\end{align*}
	for all $s\in\interval{0}{\ell_{c}}$.
	
	For any given $f\in A(\Omega)$ the function $\Phi(f)\colon\Omega\to\mathbb{C}$ defined by
	\begin{align*}
		\Phi(f)(z):=\frac{1}{2\pi i}\int_{\Omega^{\partial}}\frac{f(\sigma)^{*}}{\sigma-z} \ \mathrm{d}\sigma
	\end{align*}
	for $z\in\Omega$ has a unique continuous extension $\Phi(f)\colon\Omega^{-}\to\mathbb{C}$. In fact, $\Phi(f)\in A(\Omega)$. The induced map $\Phi\colon A(\Omega)\to A(\Omega)$ is called the \textit{conjugate Cauchy transform} on $\Omega$. Note that $\Phi$ is a bounded antilinear map. In fact, if $\Omega$ is convex, then $\Phi$ is even contractive. For a proof of these properties of $\Phi$ we refer to \cite[Lemma 2.1]{crouzeix2017numerical}. 
 
 In the special situation where $\Omega$ is a disk, the range of $\Phi$ is easily seen to consist of the constant functions, see e.g.\ \cite[Page 205]{remmert1991theory}.
	
	It follows from Cauchy's integral formula that
	\begin{align*}
		\gamma_{\Phi}(f)=\sum_{\substack{c\subset\Omega^{\partial}\\ \mathrm{connected}\\ \mathrm{component}}}\int_{0}^{\ell_{c}}f(\sigma_{c}(s))P(\sigma_{c}(s)) \ \mathrm{d}s
	\end{align*}
	for all $f\in A(\Omega)$.
 Already in Delyon-Delyon's paper \cite{delyon1999generalization} as well as in Putinar-Sandberg's \cite{putinar2005skew} and Crouzeix's paper \cite{crouzeix2007numerical}, the fact that the values of $P$ are non-negative was used. Since 
 \begin{align*}
\sum_{\substack{c\subset\Omega^{\partial}\\ \mathrm{connected}\\ \mathrm{component}}}\int_{0}^{\ell_{c}}P(\sigma_{c}(s)) \ \mathrm{d}s=2,
 \end{align*}
 this readily implies that $\|\gamma_\Phi\|\leq 2$ and thus
	\begin{align}\label{infBoundW}
		\inf_{\omega\in A(\Omega)'}\|\gamma_{\Phi}-\omega\|\leq2,
	\end{align}
	which we shall use to reprove Crouzeix--Palencia's result (Theorem \ref{CP}) result and Okubo--Ando's result (Theorem \ref{OA}) as an application of Theorem \ref{main}.

    
It turns out that von Neumann's inequality (see Theorem \ref{vN} below) also follows from Theorem \ref{main}. For this we need the following proposition due to Caldwell--Greenbaum--Li \cite[Lemma 2.1]{caldwell2018some}. For the reader's convenience, we sketch the proof. 
For any Hermitian $R$ we write $\lambda_{\minn}(R)$ for the smallest element of the (real) spectrum of $R$. 
\begin{proposition}\label{infBound}
	Let $\Omega\subset\mathbb{C}$ be a smoothly bounded open subset with conjugate Cauchy transform $\Phi\colon A(\Omega)\to A(\Omega)$. If $M$ is a bounded linear operator on a Hilbert space $H$ with spectrum inside of $\Omega$, then
	there exists an $\omega\in A(\Omega)'$ such that the functional calculus $\gamma\colon A(\Omega)\to\mathscr{L}(H)$ of $M$ on $\Omega$ satisfies 
	\begin{align*}
		\|\gamma_{\Phi}-\omega\|\leq2-\sum_{\substack{c\subset\Omega^{\partial}\\ \mathrm{connected}\\ \mathrm{component}}}\int_{0}^{\ell_{c}}\lambda_{\minn}(P(\sigma_{c}(s))) \ \mathrm{d}s.
	\end{align*}
\end{proposition}

	\begin{proof}
		We only give a sketch of the proof. For the details we refer to the original work \cite[Lemma 2.1]{caldwell2018some}. One readily verifies that the functional $\omega\colon A(\Omega)\to\mathbb{C}$ given by
		\begin{align}\label{CG-functional}
			\omega(f):=\sum_{\substack{c\subset\Omega^{\partial}\\ \mathrm{connected}\\ \mathrm{component}}}\int_{0}^{\ell_{c}}f(\sigma_{c}(s))\lambda_{\minn}(P(\sigma_{c}(s))) \ \mathrm{d}s
		\end{align}
		for $f\in A(\Omega)$ is bounded. Using that $P(z)-\lambda_{\minn}(P(z))\geq0$ for all $z\in\Omega^{\partial}$, one can prove that $\omega\in A(\Omega)'$ satisfies the desired estimate.
	\end{proof}
    Note that Theorem \ref{main} relies on the existence of extremal pairs. For finite-dimensional Hilbert spaces it is well-known that functional calculi admit extremal pairs, see \cite[Theorem 2.1]{crouzeix2004bounds}. Fortunately, for general Hilbert spaces one can often reduce to finite-dimensional ones by the following proposition, which is an extension of \cite[Theorem 2]{crouzeix2007numerical}. A \textit{finite-dimensional compression} of $M$ is a linear operator of the form $M^{\downarrow}:=\Pi M\Pi^{*}$, where $\Pi\colon H\to V$ is an orthogonal projection onto some finite-dimensional subspace $V$ of $H$.
\begin{proposition}\label{compressions}
Suppose that $\Omega$ is a smoothly bounded open subset of $\mathbb{C}$ such that the polynomials are dense in $A(\Omega)$. 
Let $\kappa$ be a positive number and $M$ a bounded linear operator on a Hilbert space $H$ with spectrum inside of $\Omega$. If $\|p(M^{\downarrow})\|\leq\kappa\|p\|_{\Omega^{\partial}}$ for every polynomial $p$ and finite-dimensional compression $M^{\downarrow}$ of $M$, then 
$\kappa$ is a spectral constant for $M$ on $\Omega$.
\end{proposition}
\begin{proof}
	Let $\gamma\colon A(\Omega)\to\mathscr{L}(H)$ denote the functional calculus of $M$ on $\Omega$. Let $p$ be a polynomial of degree $d$. Take $x\in H$ and consider the orthogonal projection $\Pi\colon H\to V$, where $V\subset H$ is the Krylov subspace spanned by the vectors $x,Mx,\ldots,M^{d}x\in H$. Define the finite-dimensional compression $M^{\downarrow}:=\Pi M\Pi^{*}$. By construction we have $\gamma(p)x=p(M^{\downarrow})x$ and therefore
	\begin{align*}
		\|\gamma(p)x\|=\|p(M^{\downarrow})x\|\leq\|p(M^{\downarrow})\|\|x\|\leq\kappa\|p\|_{\Omega^{\partial}}\|x\|.
	\end{align*}
	Since $x\in H$ was chosen arbitrarily, we obtain $\|\gamma(p)\|\leq\kappa\|p\|_{\Omega^{\partial}}$. Since we assumed that the polynomials are dense in $A(\Omega)$, we obtain $\|\gamma\|\leq\kappa$.
\end{proof}

 \begin{remark}
     Due to the use of the holomorphic functional calculus, the results on spectral sets always comes with the obstruction that the set $\Omega$ has to be open, smoothly bounded and contain the spectrum of $M$. In light of von Neumann's inequality and Crouzeix's conjecture this is unnatural in the sense that one would like to allow for spectral points on the boundary of the spectral set. However, if the spectral constants are universal and thus not depending on $\Omega$ such stronger statements follow automatically. More precisely, given a fixed bounded set $W\subset\mathbb{C}$ and a constant $\kappa\geq1$, the statement
     \begin{align*}
            \|p(M)\|\leq\kappa\|p\|_{\Omega^{\partial}},   
     \end{align*}
     for all polynomials $p$, and $\Omega\supset W$ open, convex, smoothly bounded, which contain the spectrum of $M$, implies that $W$ is a $\kappa$-spectral set for $M$.
 \end{remark}
	\subsection{Crouzeix--Palencia's Result}

    We apply our main result to recover that, if $\Omega$ is convex and contains the closure of the numerical range of $M$, then $1+\sqrt{2}$ is a spectral constant for $M$ on $\Omega$. This is was first shown by Crouzeix--Palencia \cite{crouzeix2017numerical} by combining $\|\gamma_{\Phi}\|\leq 2$ with the fact that $\|\Phi\|\leq1$ in a rather technical fashion. As observed in \cite{ransford2018remarks}, this final step can be simplified by a trick involving the formula
    \begin{align*}
    \|\gamma(f)^* \gamma(f) \gamma(f)^* \gamma(f)\|=\|\gamma(f)\|^{4}
    \end{align*}
    together with the multiplicativity of the homomorphism. The following proof is different to that approach, but instead exploits the Hilbert space structure and the homomorphism through the extremal measure.
\begin{theorem}[Crouzeix--Palencia]\label{CP}
	Suppose that $\Omega\subset\mathbb{C}$ is a convex bounded open subset with smooth boundary. If $M$ is a bounded linear operator on a Hilbert space $H$ such that $W(M)^{-}\subset\Omega$, then $1+\sqrt{2}$ is a spectral constant for $M$ on $\Omega$.
\end{theorem}
	\begin{proof}
        If $M^{\downarrow}$ is a finite-dimensional compression of $M$, then $W(M^{\downarrow})\subset W(M)$ and therefore $W(M^{\downarrow})^{-}\subset\Omega$. Hence by Proposition \ref{compressions} we can reduce to the case where $H$ is finite-dimensional.
        
		Let $\gamma\colon A(\Omega)\to\mathscr{L}(H)$ be the functional calculus of $M$ on $\Omega$. Let $\Phi\colon A(\Omega)\to A(\Omega)$ be the  conjugate Cauchy transform on $\Omega$. The convexity of $\Omega$ and the inclusion $W(M)\subset\Omega$ imply that (\ref{infBoundW}) holds. If we now use Proposition \ref{existence extremal measure} to find an extremal measure $\mu$ for an extremal pair $(f_{0},x_{0})$ of $\gamma$, then we can estimate
		\begin{align*}
			|\langle\gamma(\Phi(f_{0})f_{0})x_{0},x_{0}\rangle|&=\bigg|\int_{\Omega^{\partial}}\Phi(f_{0})(z)f_{0}(z) \ \mathrm{d}_{\mu}z\bigg|\\
   &\leq\|\Phi(f_{0})\|_{\Omega^{\partial}}\|f_{0}\|_{\Omega^{\partial}}\leq\|\Phi\|\leq1,
		\end{align*}
		where we used in the last step that $\Phi$ is contractive by convexity of $\Omega$. Thus an application of Theorem \ref{main} yields the desired result.
	\end{proof}
 \begin{remark}
     In fact, the technique used in \cite{ransford2018remarks} and later in \cite{ostermann2020abstract} shows a slightly stronger statement of Crouzeix--Palencia's result. In fact, it is proved that $1+\sqrt{2}$ is still a spectral constant for $M$ on $\Omega$ after replacing the convexity assumption on $\Omega$ and the inclusion $W(M)\subset\Omega$ by the weaker condition that there is a contractive antilinear map $\Phi\colon A(\Omega)\to A(\Omega)$ for which $\|\gamma_{\Phi}\|\leq2$ holds. This more general result can also be recovered with a proof technique similar to that of Theorem \ref{CP} under the additional assumption that an extremal pair exists.
\end{remark}
	

\begin{remark}
Suppose that $\Omega$ is convex and contains $W(M)^{-}$. Note that, if it were true that
\begin{align*}
\inf_{\omega\in A(\Omega)'}\|\gamma_{\Phi}-\omega\|\leq2-\frac{|\langle\gamma(\Phi(f_{0})f_{0})x_{0},x_{0}\rangle|}{2},
\end{align*}
then  $\|\gamma\|\leq2$ would follow by Theorem \ref{main}. This would establish Crouzeix's conjecture \cite{crouzeix2004bounds}.
\end{remark}
 
	\subsection{Okubo--Ando's Result}
	
	The original proof of the following result, due to Okubo--Ando \cite{okubo1975constants}, was based on dilation theory. A more recent proof strategy involved the estimate $\|\gamma_{\Phi}\|\leq2$ and properties of extremal pairs, see \cite{caldwell2018some}. Following the latter approach, we demonstrate how the result can be obtained from Theorem 5.
	
\begin{theorem}[Okubo--Ando]\label{OA}
	Suppose that $\Omega\subset\mathbb{C}$ is an open disk. If $M$ is a bounded linear operator on a Hilbert space $H$ such that $W(M)^{-}\subset\Omega$, then $2$ is a spectral constant for $M$ on $\Omega$.
\end{theorem}
	\begin{proof}
        If $M^{\downarrow}$ is a finite-dimensional compression of $M$, then $W(M^{\downarrow})\subset W(M)$ and therefore $W(M^{\downarrow})^{-}\subset\Omega$. Hence by Proposition \ref{compressions} we can reduce to the case where $H$ is finite-dimensional.
         
		Let $\gamma\colon A(\Omega)\to\mathscr{L}(H)$ be the functional calculus of $M$ on $\Omega$. Let $\Phi\colon A(\Omega)\to A(\Omega)$ be the conjugate Cauchy transform on $\Omega$. The convexity of $\Omega$ and the inclusion $W(M)\subset\Omega$ imply that (\ref{infBoundW}) holds. Fix an extremal pair $(f_{0},x_{0})$ for $\gamma$. Because $\Omega$ is a disk, we know that $\Phi(f_{0})$ is constant. Of course we may assume that $\|\gamma\|>1$. So $\langle\gamma(\Phi(f_{0})f_{0})x_{0},x_{0}\rangle=0$ by Proposition \ref{orthogonal}. The result now follows from Theorem \ref{main}.
\end{proof}
\begin{remark}
    Drury \cite{drury2008symbolic} was able to refine Okubo--Ando's result in the case where $\Omega$ is the open unit disk centered at the origin by showing the following: Under the conditions of Theorem \ref{OA}, it holds that 
    \begin{align*}
        \|p(M)\| \leq \nu (|p(0)|)
    \end{align*}
    for all polynomials $p$ with $\|p\|_{\Omega^{\partial}}\leq1$, where $\nu\colon\interval{0}{1}\to\interval{1}{2}$ is an explicitly given strictly decreasing function with $\nu(0)=2$ and $\nu(1)=1$. We point out that our methods do not seem to allow for reproving this stronger statement. This may not be surprising as Dury's technique strongly rests on a theorem by Berger--Stampfli \cite[Theorem 4]{berger1967mapping} stating that, if the numerical range of $M$ is included in the closed unit disk, then the same is true for $p(M)$ provided that $\|p\|_{\Omega^{\partial}}\leq1$ and $p(0)=0$. The latter result was originally proved via dilation theory, see also  \cite[Theorem 3]{putinar2005skew} for a proof relating to the double-layer potential. There also is a more recent proof of Berger--Stampfli's result using Blaschke products, \cite[Theorem 2.1]{klaja2016mapping}.
\end{remark}
\subsection{Von Neumann's Inequality}
The following well-known result, due to von Neumann \cite{neumann1950spektraltheorie}, has its origins already in 1950. The proof presented below combines Theorem \ref{main} with an idea from \cite[Lemma 6]{crouzeix2019spectral}, which in turn can be traced back to \cite{delyon1999generalization}. In what follows we write $\mathbb{T}$ for the unit circle in $\mathbb{C}$. 
\begin{theorem}[von Neumann]\label{vN}     
	Suppose that $\Omega\subset\mathbb{C}$ is an open disk centered at the origin. If $M$ is a bounded linear operator on a Hilbert space $H$ such that $\|M\|\mathbb{T}\subset\Omega$, then $1$ is a spectral constant for $M$ on $\Omega$.
\end{theorem}
\begin{proof}
    If $M^{\downarrow}$ is a finite-dimensional compression of $M$, then $\|M^{\downarrow}\|\leq\|M\|$ and therefore $\|M^{\downarrow}\|\mathbb{T}\subset\Omega$. Hence by Proposition \ref{compressions} we can reduce to the case where $H$ is finite-dimensional.
    
	Let $\gamma\colon A(\Omega)\to\mathscr{L}(H)$ be the functional calculus of $M$ on $\Omega$ and assume to the contrary that $\|\gamma\|>1$. Let $r$ denote the radius of the disk $\Omega$ and consider the arc length parametrization $\sigma\colon\interval{0}{2\pi r}\to\mathbb{C}$ given by $\sigma(s):=re^{is/r}$ for $s\in\interval{0}{2\pi r}$. Let $\Phi\colon A(\Omega)\to A(\Omega)$ be the conjugate Cauchy transform on $\Omega$. Fix an extremal pair $(f_{0},x_{0})$ for $\gamma$. Since $\Phi(f_{0})$ is constant and $\|\gamma\|>1$, we have $\langle\gamma(\Phi(f_{0})f_{0})x_{0},x_{0}\rangle=0$ by Proposition \ref{orthogonal}. Moreover, an elementary computation shows that
	\begin{align*}
		2\pi rP(\sigma(s))-1=(\sigma(s)^{*}-M^{*})^{-1}(r^{2}-M^{*}M)(\sigma(s)-M)^{-1}\geq0
	\end{align*}
	for all $s\in\interval{0}{2\pi r}$. It now follows from Theorem \ref{main} and Proposition \ref{infBound} that
	\begin{align*}
		\|\gamma\|&\leq\inf_{\omega\in A(\Omega)'}\|\gamma_{\Phi}-\omega\|\\
  &\leq2-\int_{0}^{2\pi r}\lambda_{\minn}(P(\sigma(s))) \ \mathrm{d}s\\
  &\leq2-\int_{0}^{2\pi r}\frac{1}{2\pi r} \ \mathrm{d}s=1,
	\end{align*}
	which contradicts the assumption $\|\gamma\|>1$. Hence we must have $\|\gamma\|=1$.
\end{proof}
\subsection{Crouzeix--Greenbaum's Result}
Using an extremal measure $\mu$ for an extremal pair $(f_{0},x_{0})$ of $\gamma$ we find that
\begin{align*}
   |\langle\gamma(\Phi(f_{0})f_{0})x_{0},x_{0}\rangle|\leq\|\Phi\|,
\end{align*} 
see the proof of Theorem \ref{CP} for a similar reasoning. This directly yields the following weaker version of Theorem \ref{main} to reprove a statement from Crouzeix--Greenbaum \cite{crouzeix2019spectral}.
\begin{corollary}\label{CG}
Let $\Omega\subset\mathbb{C}$ be a smoothly bounded open subset and $\Phi\colon A(\Omega)\to A(\Omega)$ a bounded antilinear map. Suppose that $M$ is a bounded linear operator on a Hilbert space $H$ with spectrum inside of $\Omega$. Let $\gamma\colon A(\Omega)\to\mathscr{L}(H)$ denote the functional calculus of $M$ on $\Omega$ and suppose that it admits an extremal pair. Moreover, let $\omega\colon A(\Omega)\to\mathbb{C}$ be any bounded linear functional. If $\|\gamma\|>1$, then
\begin{align*}
\|\gamma\|\leq\|\gamma_{\Phi}-\omega\|+\sqrt{\|\gamma_{\Phi}-\omega\|^{2}+\|\Phi\|+\|\omega\|}.
\end{align*}
\end{corollary}
Apart from the assumption that $\gamma$ admits an extremal pair,  Corollary \ref{CG} is precisely \cite[Theorem 2]{crouzeix2019spectral}. Our main result shows that the term $\|\omega\|$ can be dropped and seems to be an artifact of the proof technique used in \cite{crouzeix2019spectral}. However, in most applications presented in \cite{crouzeix2019spectral} the functional $\omega$ is chosen to be $0$ anyway.

\section{Relation to Numerical Radii}
 The \textit{numerical radius} of a bounded linear operator on $M$ on $H$ is given by
\begin{align*}
    w(M):=\sup_{z\in W(M)}|z|
\end{align*}
 and defines a norm on $\mathscr{L}(H)$ that is known to be equivalent to the operator norm on $\mathscr{L}(H)$. Recall our general framework where $A$ is a unital uniform algebra acting on a compact Hausdorff space $K$ and $\gamma\colon A\to\mathscr{L}(H)$ a unital bounded homomorphism. Badea--Crouzeix--Klaja \cite{badea2018spectral}, see also \cite{davidson2018complete}, it was shown that the induced numerical radius operator norm of $\gamma$, which we denote by $\|\gamma\|_{w}$, is explicitly linked to $\|\gamma\|$ via
\begin{align*}
    2\|\gamma\|_{w}=\frac{1}{\|\gamma\|}+\|\gamma\|.
\end{align*}
Hence bounds on $\|\gamma\|_{w}$ give bounds on $\|\gamma\|$ and vice versa. Also note that an analogous notion of extremal pairs and associated measure can be defined in this setting, see \cite{bickel2020crouzeix}. A variant of Proposition \ref{existence extremal measure}  using this different norm can be proved easily in a similar fashion.
\bibliographystyle{alpha}
\bibliography{Bibliography1}

\end{document}